\documentclass[12pt,a4paper]{amsart}

\usepackage[margin=2.5cm]{geometry}
\usepackage[dvips]{graphicx}
\usepackage{amssymb}
\usepackage{amsmath}
\usepackage{amscd}
\usepackage[all]{xy}

\author{Sergiy Maksymenko}
\title[Functions with isolated singularities, II]
{Functions with isolated singularities \\ on surfaces, II}
\address{Topology dept., Institute of Mathematics of NAS of Ukraine, Te\-re\-shchen\-kivska st. 3, Kyiv, 01601 Ukraine}
\email{maks@imath.kiev.ua}
\urladdr{http://www.imath.kiev.ua/~maks}

\thanks{This research is partially supported by grants of Ministry of Science and Education of Ukraine No~M/150-2009, and by grant of The State Fund of Fundamental Researches of Ukraine and Russian Foundation for Basic Research, No.~$\Phi$40.1/009.}

\makeatletter
\newcommand\testshape{family=\f@family; series=\f@series; shape=\f@shape.}
\def\myemphInternal#1{\if n\f@shape%
\begingroup\itshape #1\endgroup\/%
\else\begingroup\bfseries #1\endgroup%
\fi}
\def\myemph{\futurelet\testchar\MaybeOptArgmyemph}
\def\MaybeOptArgmyemph{\ifx[\testchar \let\next\OptArgmyemph
                 \else \let\next\NoOptArgmyemph \fi \next}
\def\OptArgmyemph[#1]#2{\index{#1}\myemphInternal{#2}}
\def\NoOptArgmyemph#1{\myemphInternal{#1}}
\makeatother

\newtheorem{theorem}[subsection]{Theorem}
\newtheorem{lemma}[subsection]{Lemma}

\newtheorem{corollary}[subsection]{Corollary}

\newtheorem{remark}[subsection]{Remark}

\newtheorem{definition}[subsection]{Definition}

\newenvironment{axiom}[1]
{
\medskip\par\noindent
{\bf Axiom #1.}\begin{it}
}
{
\end{it}
\par
\medskip
 }

\makeatletter
\@addtoreset{equation}{section}
\@addtoreset{figure}{section}
\@addtoreset{table}{section}
\makeatother

\newcommand\CCC{{\mathbb C}}

\newcommand\RRR{{\mathbb R}}
\newcommand\ZZZ{{\mathbb Z}}

\newcommand\FF{{\mathcal F}}

\newcommand\id{\mathrm{id}}

\newcommand\Int{\mathrm{Int}}

\newcommand\supp{\mathrm{supp\,}}

\newcommand\Orbit{\mathcal{O}}
\newcommand\Stab{\mathcal{S}}
\newcommand\Diff{\mathcal{D}}

\newcommand\Cont[1]{\mathcal{C}^{#1}}
\newcommand\Cr[3]{\Cont{#1}(#2,#3)}
\newcommand\Ci[2]{\Cr{\infty}{#1}{#2}}

\newcommand\Cinf{\mathcal{C}^{\infty}}

\newcommand\Circle{S^1}
\newcommand\Mman{M}
\newcommand\Nman{N}
\newcommand\Pman{P}
\newcommand\Sman{S}
\newcommand\Uman{U}

\newcommand\Wman{W}
\newcommand\Xman{X}
\newcommand\Yman{Y}

\newcommand\tYman{\widetilde{\Yman}}
\newcommand\tXman{\widetilde{\Xman}}
\newcommand\tNman{\widetilde{\Nman}}

\newcommand\tMman{\widetilde{\Mman}}
\newcommand\FldA{F}
\newcommand\FlowA{\mathbf{\FldA}}
\newcommand\func{f}
\newcommand\afunc{\alpha}
\newcommand\bfunc{\beta}

\newcommand\gfunc{g}
\newcommand\tfunc{\widetilde{\func}}

\newcommand\singf{\Sigma_{\func}}
\newcommand\partitf{\Delta_{\func}}

\newcommand\DiffM{\Diff(\Mman)}
\newcommand\DiffMX{\Diff(\Mman,\Xman)}
\newcommand\DiffId{\Diff_{\id}}
\newcommand\DiffIdMX{\DiffId(\Mman,\Xman)}
\newcommand\DiffIdM{\DiffId(\Mman)}
\newcommand\tDiff{\widetilde{\Diff}}

\newcommand\tDifftM{\widetilde{\Diff}(\tMman)}
\newcommand\tDiffIdtM{\widetilde{\Diff}_{\id}(\tMman)}

\newcommand\tDifftMtX{\tDiff(\tMman,\tXman)}

\newcommand\tStab{\widetilde{\Stab}}
\newcommand\StabId{\Stab_{\id}}
\newcommand\tStabId{\tStab_{\id}}

\newcommand\StabIdf{\Stab_{\id}(\func)}

\newcommand\StabfX{\Stab(\func,\Xman)}
\newcommand\StabIdfX{\StabId(\func,\Xman)}

\newcommand\tStabtf{\tStab(\tfunc)}
\newcommand\tStabIdtf{\tStabId(\tfunc)}
\newcommand\tStabIdtftX{\tStabId(\tfunc,\tXman)}

\newcommand\Orbf{\Orbit(\func)}
\newcommand\Orbff{\Orbit_{\func}(\func)}
\newcommand\OrbfX{\Orbit(\func,\Xman)}
\newcommand\OrbffX{\Orbit_{\func}(\func,\Xman)}

\newcommand\ShA{\varphi}

\newcommand\DoubleCover{\beta}
\newcommand\Invol{\xi}

\newcommand\InitHom{\psi}
\newcommand\InitLift{\kappa}
\newcommand\XInitLift{\lambda}
\newcommand\ReqLift{\eta}
\newcommand\ShFInitLift{\delta}
\newcommand\DefShFInitLift{\Delta}

\newcommand\DefInitLift{K}

\newcommand\MobiusBand{M\!\text{\"o}}
\newcommand\prjplane{\RRR\mathrm{P}^2}
\newcommand\Kleinb{\mathbb{K}}

\newcommand\dif{h}
\newcommand\tdif{\tilde{\dif}}

\newcommand\AxBd{{\text{\rm(B1)}}}
\newcommand\AxCrPt{{\text{\rm(L1)}}}
\newcommand\AxVF{{\text{\rm(B2)}}}
\newcommand\AxFibr{{\text{\rm(B3)}}}
\newcommand\AxRestr{{\text{\rm(B4)}}}

\begin{document}
\begin{abstract}
Let $M$ be a smooth connected compact surface, $P$ be either the real line
$\mathbb{R}$ or the circle $S^1$.
For a subset $X\subset M$ denote by $\mathcal{D}(M,X)$ the group of diffeomorphisms of $M$ fixed on $X$.
In this note we consider a special class $\mathcal{F}$ of smooth maps $f:M\to P$ with isolated singularities which includes all Morse maps.
For each map $f\in \mathcal{F}$ we consider certain submanifolds $X\subset M$ that are ``adopted'' with $f$ in a natural sense, and study the right action of the group $\mathcal{D}(M,X)$ on $C^{\infty}(M,P)$.
The main result describes the homotopy types of the connected components of the stabilizers $\mathcal{S}(f)$ and orbits $\mathcal{O}(f)$ for all maps $f\in \mathcal{F}$.
It extends previous author results on this topic.
\end{abstract}

\maketitle

\section{Introduction}
Let $\Mman$ be a smooth compact connected surface and $\Pman$ be either the real line $\RRR$ or the circle $\Circle$.
In this paper we study the subspace $\FF\subset\Ci{\Mman}{\Pman}$ consisting of maps $\func:\Mman\to\Pman$ satisfying the following two axioms:

\begin{axiom}{\AxBd}
The set $\singf$ of critical points of $\func$ is finite and is contained in the interior $\Int{\Mman}$, and $\func$ takes a constant value at each boundary component of $\Mman$.
\end{axiom}
\begin{axiom}{\AxCrPt}
For every critical point $z$ of $\func$ there exists a local presentation $\func_{z}:\RRR^2\to\RRR$ of $\func$ in which $z=(0,0)$ and $\func_{z}$ is a homogeneous polynomial without multiple factors.
\end{axiom}
For instance, due to Morse lemma each non-degenerate critical point of a function $\func:\Mman\to\Pman$ is equivalent to a homogeneous polynomial $\pm x^2\pm y^2$ having no multiple factors.
Hence each Morse function satisfies axiom \AxCrPt.

Recall that every homogeneous polynomial $\gfunc:\RRR^2\to\RRR$ can be expressed as a product $\gfunc = L_1^{p_1}\cdots L^{p_{\alpha}}_{\alpha} Q^{q_1}_1 \cdots Q^{q_{\beta}}_{\beta}$, where $L_i(x,y)=a_i x + b_i y$, and $Q_j(x,y)= c_j x^2 + 2d_j xy + e_j y^2$ is an irreducible over $\RRR$ (definite) quadratic form, $L_i/L_{i'}\not=\mathrm{const}$ for $i\not=i'$, and $Q_j/Q_{j'}\not=\mathrm{const}$ for $j\not=j'$.
Then Axiom \AxCrPt\ requires that $p_i=q_j=1$ for all $i,j$.

Notice that if $p_i\geq2$ for some $i$, then the line $\{L_i=0\}$ consists of critical points of $\func$, whence Axiom \AxCrPt\ implies that all critical points of $\func$ are isolated.
Moreover, the requirement that $q_j=1$ for all $j$ is a certain non-degeneracy assumption.

\begin{definition}
Let $\Xman \subset \Mman$ be a compact submanifold such that its connected components may have distinct dimensions.
Denote by $\Xman^i$, $i=0,1,2$, the union of connected components of $\Xman$ of dimension $i$.
Let also $\func:\Mman\to\Pman$ be a smooth map satisfying axiom \AxBd.
We will say that $\Xman$ is an \myemph{$\func$-adopted} if the following conditions hold true:
\begin{enumerate}
 \item[(0)] $\Xman^0\subset \singf$;
 \item[(1)] $\Xman^1\cap\singf=\varnothing$ and $\func$ takes constant value on each connected component of $\Xman^1$;
 \item[(2)] the restriction $\func|_{\Xman^2}$ satisfies axiom \AxBd\ as well.
\end{enumerate}
\end{definition}

For instance, the following sets and their connected components are \emph{$\func$-adopted}: $\varnothing$, $\partial\Mman$, $\singf$, $\func^{-1}(c)$, where $c\in\Pman$ is a regular value of $\func$, $\func^{-1}(I)$, where $I\subset\Pman$ is a closed interval whose both ends are regular values of $\func$.

\medskip

Let $\Xman\subset\Mman$ be an $\func$-adopted submanifold, and $\DiffMX$ be the group of diffeomorphisms of $\Mman$ fixed on $\Xman$.
Endow $\DiffMX$ and $\Ci{\Mman}{\Pman}$ with $\Cinf$-topologies.
Then $\DiffMX$ continuously acts from the right on $\Ci{\Mman}{\Pman}$ by the formula:
\begin{equation}\label{equ:DiffM_action1}
\func \cdot \dif = \func \circ \dif, \qquad \dif\in\DiffMX, \ \func\in\Ci{\Mman}{\Pman}.
\end{equation}
For $\func\in\Ci{\Mman}{\Pman}$ let $\StabfX=\{\dif\in\DiffMX \mid \func\circ\dif=\func\}$ and $\OrbfX=\{\func\circ\dif\mid \dif\in\DiffMX\}$ be respectively the stabilizer and the orbit of $\func$.
Let also $\DiffIdMX$ and $\StabIdfX$ be the identity path components of $\DiffMX$ and $\StabfX$, and $\OrbffX$ be the path component of $\func$ in $\OrbfX$.

We will omit notation for $\Xman$ whenever it is empty, for instance $\StabIdf = \StabId(\func,\varnothing)$, and so on.

In a series of papers~\cite{Maks:AGAG:2006, Maks:TrMath:2008, Maksymenko:ProcIM:ENG:2010, Maks:MFAT:2010} for the cases $\Xman=\varnothing$ and $\Xman=\singf$ the author calculated the homotopy types of $\StabIdfX$ and $\OrbffX$ for a large class of smooth maps $\func:\Mman\to\Pman$ which includes all maps satisfying axioms \AxBd\ and \AxCrPt.

The aim of this paper is to extend these results to the general case of $\DiffMX$, where $\Xman$ is an $\func$-adopted submanifold, see Section~\ref{sect:main_results}.

\subsection{Notation.}
Throughout the paper $T^2$ will be a $2$-torus $S^1\times S^1$, $\MobiusBand$ a M\"obius band, and $\Kleinb$ a Klein bottle.
For topological spaces $X$ and $Y$ the notation $X\cong Y$ will mean that $X$ and $Y$ are homotopy equivalent.

For a map $\func:\Mman\to\Pman$ and $c\in\Pman$ the set $\func^{-1}(c)$ will be called a \myemph{level set} of $\func$.
Let $\omega$ be connected component of some level set of $\func$.
Then $\omega$ is \myemph{critical} if it contains a critical point of $\func$.
Otherwise $\omega$ will be called \emph{regular}.

We will denote by $\partitf$ the partition on $\Mman$ whose elements are critical points of $\func$ and connected components of the sets $\func^{-1}(c)\setminus\singf$ for all $c\in\Pman$, see~\cite{Maks:AGAG:2006, Maksymenko:ProcIM:ENG:2010}.

For a vector field $\FldA$ on $\Mman$ and a smooth function $\afunc:\Mman\to\RRR$ we will denote by $\FldA(\afunc)$ the Lie derivative of $\afunc$ along $\FldA$.

\subsection{Acknowledgements}
The author would like to thank Professor Tatsuhiko Yagasaki for useful discussions of the homotopy types of $\DiffIdMX$ for compact surfaces.

\section{Main results}\label{sect:main_results}
In this section we will assume that $\func:\Mman\to\Pman$ is a smooth map satisfying axioms \AxBd\ and \AxCrPt, and $\Xman\subset\Mman$ is an $\func$-adopted submanifold.
The principal results of this paper are Theorems~\ref{th:StabIdf}, \ref{th:DiffMX_to_OrbfX_fibr}, \ref{th:OrbffX}, and~\ref{th:compute_pi1Of}.
They are new only for the case when {\em $\Xman$ is infinite}.

\medskip

\noindent
{\bf Homotopy type of $\StabIdfX$.}
By~\cite[Th.~1.9]{Maks:AGAG:2006} and~\cite[Th.~3]{Maksymenko:ProcIM:ENG:2010} $\StabIdf$ is contractible except for the functions of types (A)-(D) of~\cite[Th.~1.9]{Maks:AGAG:2006}:
\begin{itemize}
\item[(A)]
$\func_A:S^2\to\Pman$ with only two non-degenerate critical points: one maximum and one minimum, and both points are non-degenerate.

\item[(B)] 
$\func_B:D^2\to\Pman$ with a unique critical point being a non-degenerate local extreme;

\item[(C)] 
$\func_C:S^1\times I\to \Pman$ without critical points.

\item[(D)]
$\func_D:T^2 \to S^1$ without critical points.
\end{itemize}
For these functions $\StabIdf$ is homotopy equivalent to $S^1$.

The following theorem describes the homotopy type of $\StabIdfX$.

\begin{theorem}\label{th:StabIdf}{\rm c.f.~\cite{Maks:AGAG:2006, Maksymenko:ProcIM:ENG:2010}.}
$\StabIdfX\cong S^1$ if and only if the following two conditions hold true
\begin{itemize}
 \item[\rm(i)]
$\StabIdf\cong S^1$, so $\func$ is of one of the types {\rm(A)-(D)} above, and
 \item[\rm(ii)] 
$\Xman\subset\singf$.
\end{itemize}
In all other cases $\StabIdfX$ is contractible.
\end{theorem}

The proof of this theorem will be given in Section~\ref{sect:proof_th_1}.

\medskip

\noindent
{\bf Homotopy type of $\OrbffX$.}
We will now show that for description of the homotopy type of orbits $\OrbffX$ one can always assume that $\partial\Mman\subset\Xman$.
First we need the following technical result.

\begin{theorem}\label{th:DiffMX_to_OrbfX_fibr}
The map $p:\DiffMX\to\OrbfX$ defined by $p(\dif) = \func\circ\dif$ for $\dif\in\DiffMX$ is a Serre fibration.
\end{theorem}
For $\Xman=\varnothing$ the orbit $\Orbff$ has a ``finite codimension'' in the space of all smooth functions and the result is proved in~\cite{Sergeraert:ASENS:1972}, see~\cite[Lm.~11]{Maksymenko:ProcIM:ENG:2010} for detailed explanations.
We will deduce the general case of Theorem~\ref{th:DiffMX_to_OrbfX_fibr} from the case $\Xman=\varnothing$, see Section~\ref{sect:proof_th_2}.

\begin{corollary}\label{equ:OffX_OffY}
Let $\Yman$ be a union of some connected components of $\partial\Mman$, so $\Xman\cup\Yman$ is an $\func$-adopted submanifold of $\Mman$.
Then $\Orbit_{\func}(\func,\Xman\cup\Yman) = \OrbffX$.
\end{corollary}
\begin{proof}
We can assume that $\Xman\cap\Yman=\varnothing$, otherwise just replace $\Yman$ with $\Yman\setminus\Xman$.

Evidently, $\Orbit_{\func}(\func,\Xman\cup\Yman) \subset \OrbffX$.

Conversely, let $\gfunc\in\OrbffX$, so there exists a path $\omega:I\to\OrbfX$ such that $\omega_0=\func$, and $\omega_1=\gfunc$.
Since $p$ is a Serre fibration, this path lifts to a path $\widetilde{\omega}:I\to\DiffMX$ such that $\widetilde{\omega}_0=\id_{\Mman}$ and $\omega_t = \func\circ \widetilde{\omega}_t$.
In particular, $\gfunc = \omega_1 = \func\circ\widetilde{\omega}_1$.

Since $\widetilde{\omega}_1$ is isotopic to $\id_{\Mman}$ relatively to $\Xman$, we have that it preserves each connected component of $\Yman$.
Then due to \AxBd, it is easy to construct a diffeomorphism $\dif\in\DiffMX$ such that $\dif=\widetilde{\omega}_1$ even on some neighbourhood of $\Yman$ and $\func\circ\dif=\func$.
Hence $\dif^{-1}\circ\widetilde{\omega}_1 \in \Diff(\Mman, \Xman\cup\Yman)$, and 
$\gfunc = \func\circ\widetilde{\omega}_1 = \func\circ\dif^{-1} \widetilde{\omega}_1 \in \Orbit_{\func}(\func,\Xman\cup\Yman)$.
\end{proof}

\begin{lemma}\label{lm:hom_type_DiffIdMX}
{\rm\cite{Smale:ProcAMS:1959, Birman:2:CPAM:1969, EarleEells:DG:1970, EarleSchatz:DG:1970, Gramain:ASENS:1973}.}
The homotopy types of $\DiffIdMX$ are presented in the following table:
\begin{center}
\begin{tabular}{|c|l|c|} \hline
Case & $(\Mman,\Xman)$ & Homotopy type of $\DiffIdMX$ \\  \hline  \hline
$1)$ & $S^2$, $\prjplane$ & $SO(3)$ \\ \hline
$2)$ & $T^2$ & $T^2$ \\ \hline
$3)$ & $(S^2,*)$, $(S^2,**)$, $\Kleinb$ & $S^1$ \\
$4)$ & $(D^2,*)$, $D^2$, $S^1\times I$, $\MobiusBand$  &  \\ \hline
$5)$ & all other cases & point \\ \hline
\end{tabular}
\end{center}
Here $*$ is a point; $(S^2,**)$ means that $\Xman$ consists of two points; and we omit $\Xman$ when it is empty, e.g. $S^2 = (S^2,\varnothing)$.

In particular, $\chi(\Mman)<\#(\Xman)$, e.g., when $\Xman$ is infinite, then $\DiffIdMX$ is contractible.
\qed
\end{lemma}

\begin{remark}
In the cases {\rm3)} and {\rm4)} the homotopy types of $\DiffIdMX$ are the same, but we separate these cases with respect to the existence of boundary.
So in the case {\rm3)} $\Mman$ is closed while in the case {\rm4)} $\partial\Mman\not=\varnothing$.
\end{remark}
Denote 
\[
 \Stab'(\func,\Xman) := \StabfX \cap \DiffIdMX.
\]

Thus each $\dif\in\Stab'(\func,\Xman)$ preserves $\func$ and is isotopic to $\id_{\Mman}$, though that isotopy is not assumed to be $\func$-preserving.
This group plays an important role for the fundamental group $\pi_1\OrbffX$.
Notice that $\pi_0\Stab'(\func,\Xman)$ can be regarded as the kernel of the homomorphism:
\[ 
i_0: \pi_0\StabfX \to \pi_0\DiffMX.
\]
induced by the inclusion $i:\StabfX\subset\DiffMX$.

\begin{theorem}\label{th:OrbffX}
We have that 
\begin{equation}\label{equ:pinOf_n2}
\pi_n\OrbffX=\pi_n\DiffIdMX, \qquad  n\geq2.
\end{equation}
Thus if $(\Mman,\Xman)=(S^2,\varnothing)$ or $(\prjplane,\varnothing)$, then $\pi_n\OrbffX=\pi_n S^2$, $n\geq3$, and $\pi_2\OrbffX=0$.
Otherwise, $\pi_n\OrbffX=0$, $n\geq2$, i.e. $\OrbffX$ is aspherical.

Moreover, for $\pi_1\OrbffX$ we have the following exact sequence
\begin{equation}\label{equ:pi1Of_gen_case}
 0 \to \frac{\pi_1\DiffIdMX}{\pi_1\StabIdfX} \xrightarrow{~p_1~} \pi_1\OrbffX \xrightarrow{~\partial_1~} 
 \pi_0\Stab'(\func,\Xman) \to 0.
\end{equation}

In particular, in the case {\rm5)} of Lemma~{\rm\ref{lm:hom_type_DiffIdMX}}, when $\DiffIdMX$ is contractible, we have an isomorphism
\begin{equation}\label{equ:pi1Of_case5}
\pi_1\OrbffX \approx \pi_0\Stab'(\func,\Xman).
\end{equation}

In the case {\rm4)} denote $\Yman=\Xman\cup\partial\Mman$, then 
\begin{equation}\label{equ:pi1Of_case4}
\pi_1\OrbffX = \pi_1\Orbit(\func,\Yman) \approx \pi_0\Stab'(\func,\Yman).
\end{equation}
\end{theorem}
\begin{proof}
\eqref{equ:pinOf_n2}.
Suppose that $\StabIdfX$ is contractible.
Then from the exact sequence of homotopy groups of the fibration $p:\DiffMX\to\OrbfX$ we obtain that $\pi_n\OrbffX=\pi_n\DiffIdMX$ for $n\geq2$.

Now let $\StabIdfX\cong S^1$.
Then again $\pi_n\OrbffX=\pi_n\DiffIdMX$ for $n\geq3$, while for $n=2$ we get the following part of exact sequence:
\[
  0 \to \pi_2\DiffIdMX \xrightarrow{~p_2~} \pi_2\OrbffX \xrightarrow{~\partial_2~} \pi_1\StabfX \xrightarrow{~i_1~} \pi_1\DiffIdMX 
\]
In the proof of~\cite[Th.~1.9]{Maks:AGAG:2006} it was shown that the map $i_1$ is a monomorphism, so $\pi_2\OrbffX=\pi_2\DiffIdMX$ as well.
Exact values of groups $\pi_n\DiffIdMX$ follow from Lemma~\ref{lm:hom_type_DiffIdMX}.

\eqref{equ:pi1Of_gen_case}.
Since $\pi_2\OrbffX=0$, we have the following exact sequence:
\begin{multline*}
 0 \to \pi_1\StabIdfX \xrightarrow{~i_1~} \pi_1\DiffIdMX \xrightarrow{~p_1~} \pi_1\OrbffX \xrightarrow{~\partial_1~} \\
   \xrightarrow{~\partial_1~}\pi_0\StabfX \xrightarrow{~i_0~} \pi_0\DiffMX,
\end{multline*}
which implies~\eqref{equ:pi1Of_gen_case}.

Finally~\eqref{equ:pi1Of_case5} follows from~\eqref{equ:pi1Of_gen_case}, and~\eqref{equ:pi1Of_case4} from~\eqref{equ:pi1Of_gen_case} and Corollary~\ref{equ:OffX_OffY}.
\end{proof}

\medskip

\noindent{\bf Fundamental group $\pi_1\OrbffX$.}
The following theorem shows that the computations of $\pi_1\OrbffX$ almost always reduces to the case when $\Mman$ is either $D^2$, or $S^1\times I$, or $\MobiusBand$.
It extends \cite[Th.~1.8]{Maks:MFAT:2010} to the case when $\Xman$ is infinite.

\begin{theorem}\label{th:compute_pi1Of}{\rm c.f.~\cite[Th.~1.8]{Maks:MFAT:2010}.}
Suppose one the following conditions holds true:
\[
\text{{\rm(i)}~$\partial\Mman\not=\varnothing$;
\qquad 
{\rm(ii)}~$\chi(\Mman)<0$;
\qquad 
{\rm(iii)}~$\Xman$ is infinite.}
\]

Then there exist finitely many $\func$-adapted mutually disjoint compact subsurfaces $B_1,\ldots, B_n$ with the following properties:
\begin{itemize}
  \item 
$\Int B_i \cap \Xman\subset \Xman^0$;
  \item 
each $B_i$ is diffeomorphic either to $D^2$, or $S^1\times I$, or $\MobiusBand$;
  \item 
put $\Yman_i = \partial B_i \cup (B_i\cap \Xman^0)$, then 
\begin{equation}\label{equ:pi1Of_prod_pi1OfdBi}
\pi_1\OrbffX \approx 
\prod_{i=1}^{n} \pi_0\Stab'(\func|_{B_i}, \Yman_i).
\end{equation}
\end{itemize}
\end{theorem}

The rest of the paper is devoted to proof of Theorems~\ref{th:StabIdf}, \ref{th:DiffMX_to_OrbfX_fibr}, and~\ref{th:compute_pi1Of}. 

\section{Proof of Theorem~\ref{th:compute_pi1Of}}
\label{sect:proof:th:compute_pi1Of}
The proof will be given at the end of this section and now we will establish one technical result.

Let $\func:S^1\times I\to I$ be the function defined by $\func(z,\tau)=\tau$.
For a non-empty subset $A\subset I$ denote by $\Stab_{A}$ the stabilizer $\Stab\bigl(\func, S^1\times A\bigr)$, i.e.\! the group of diffeomorphisms $\dif$ of $S^1\times I$ such that 
\begin{enumerate}
 \item
$\func\circ\dif=\func$, so $\dif(S^1\times \tau) = S^1\times \tau$ for all $\tau$;
 \item
$\dif$ is fixed on $S^1\times A$.
\end{enumerate}
Let $J \subset I=[0,1]$ be a \myemph{non-empty}, \myemph{closed}, and \myemph{connected} subset, $T$ be a closed neighbourhood of $J$ in $I$, and $T'$ be a closed neighbourhood of $T$ in $I$.
\begin{lemma}\label{lm:deform_near_X}
The inclusion $i:\Stab_{T} \subset \Stab_{J}$ is a homotopy equivalence, so there exists a homotopy  
\[
 H:\Stab_{J} \times [0,1] \longrightarrow \Stab_{J},
\]
such that $H_0=\id(\Stab_{J})$ and $H_1(\Stab_{J})\subset \Stab_{T}$.
Moreover, $H_{s}(\dif)=\dif$ on $S^1\times(I\setminus T')$ for all $s\in [0,1]$ and $\dif\in\Stab_{J}$, and in particular, $H_s$ is fixed on $\Stab_{T'}$.
\end{lemma}
\begin{proof}
Identify $S^1$ with the unit circle in the complex plane $\CCC$ and define the following vector field $\FldA(z,\tau)=\frac{\partial}{\partial z}$ on $S^1\times I$ generating the flow
\[\FlowA:(S^1\times I)\times\RRR\to S^1\times I,
\qquad 
\FlowA(z,\tau,t) = (e^{2\pi i \tau} z, \tau).
\]
We claim that there exists a unique map $\Delta: \Stab_{J} \to \Ci{S^1\times I}{\RRR}$ continuous with respect to $\Cinf$-topologies and such that
\begin{enumerate}
 \item[(a)]
$\dif(z,\tau) = \FlowA\bigr(z,\tau, \Delta(\dif)(z,\tau)\bigr) 
= \Bigl( e^{2\pi i \cdot \Delta(\dif)(z,\tau)}z, \tau\Bigr)$.

 \item[(b)]
Let $Y\subset I$ be any closed connected subset containing $J$.
Then $\Delta(\dif)(z,\tau) = 0$ for $(\dif,\tau)\in \Stab_{Y}\times Y$.
In particular, $\Delta(\dif)(z,\tau) = 0$ for all $\tau\in J$.
\end{enumerate}
Indeed, let $Q:\RRR\times I\to S^1\times I$, $Q(t,\tau) = (e^{2\pi i t},\tau)$ be the universal covering map of $S^1\times I$.
Then each $\dif\in\Stab_{J}$ lifts to a unique map
\[
 \tdif=(\tdif_1,\tdif_2):\RRR\times I \to \RRR\times I
\]
such that $\dif\circ Q=Q\circ\tdif$ and $\tdif$ is fixed on $Q^{-1}(S^1\times J)$.
Put
\[
 \Delta(\dif)(t,\tau) = \tdif_1(t,\tau)-t.
\]
We claim that $\Delta$ satisfies conditions (a) and (b) above.

(a) 
Notice that
\begin{align*}
 Q\circ\tdif(t,\tau) &= \left(e^{2\pi i \tdif_1},\tdif_2\right),
&
 \dif\circ Q(t,\tau) &= \dif(e^{2\pi i t},\tau).
\end{align*}
Then from the the identity $\dif\circ Q=Q\circ\tdif$ we get
\begin{align*}
 \dif(z,\tau) &= \dif(e^{2\pi i t},\tau) = 
\left(e^{2\pi i \tdif_1},\tau\right) =
\left(e^{2\pi i t} \cdot e^{2\pi i[\tdif_1(t,\tau)-t]},\tau\right) = \\ &=
\left(e^{2\pi i\cdot \Delta(\dif)(t,\tau)} z,\tau\right) =
\FlowA\bigr(z,\tau, \Delta(\dif)(z,\tau)\bigr).
\end{align*}

(b) Let $\tau\in Y$ and $\dif\in\Stab_{Y}\subset \Stab_{J}$.
Since $\dif$ is fixed on $S^1\times Y$ and $Y$ is connected, it follows that the lifting $\tdif$ is fixed on $\RRR\times Y$, i.e. $\tdif(t,\tau)=(t,\tau)$ for all $(t,\tau)\in\RRR\times Y$.
This means that $\tdif_1(t,\tau)=t$, whence $\Delta(\dif)(t,\tau) = t-t=0$.

\medskip

Now fix any $\Cinf$-function $\mu:I\to[0,1]$ such that $\mu=0$ on $T$ and $\mu=1$ on $\overline{I\setminus T'}$, and defined the homotopy 
$H:\Stab_{J} \times [0,1] \longrightarrow \Stab_{J}$,
by
\[
 H(\dif, s) = \FlowA\bigr(z,\tau, (s\mu(\tau)+1-s)\cdot\Delta(\dif)(z,\tau)\bigr).
\]
We claim that $H$ satisfies statement of lemma.

1) First notice that $H_0=\id$.
Indeed,
\[
 H(\dif, 0)(z,\tau) = \FlowA\bigr(z,\tau, \Delta(\dif)(z,\tau)\bigr) 
\stackrel{~(a)~}{=\!=\!=} \dif(z,\tau).
\]

2) $H(\dif,s)$ is fixed on $S^1\times J$.
Indeed, if $(\tau\in J$, then $\Delta(\dif)(z,\tau)=0$, whence
\[
 H(\dif, s)(z,\tau) = \FlowA\bigr(z,\tau, 0\bigr)=(z,\tau).
\]

3) Let us verify that $H(\dif,s)$ is a diffeomorphism.
Notice $H(\dif,s)$ is obtained by substitution of a smooth function $\alpha=(s\mu+1-s)\cdot\Delta(\dif)$ into the flow map instead of time.
Then, due to~\cite{Maks:TA:2003}, $H(\dif,s)$ is a diffeomorphism if and only if the Lie derivative 
\begin{equation}\label{equ:Shift_is_diff}
\FldA(\alpha)>-1.
\end{equation}
In particular, since $\dif=H(\dif,0)$ is a diffeomorphism we have that $\FldA(\Delta(\dif))>-1$.
Hence 
\[
 \FldA\bigl( (s\mu+1-s)\cdot\Delta(\dif) \bigr) =
 \FldA(s\mu+1-s)\cdot\Delta(\dif) + (s\mu+1-s) \FldA\bigl(\Delta(\dif)\bigr).
\]
Since $\mu$ depends only on $\tau$, we obtain that $\FldA(s\mu+1-s)=0$, and so the first summand vanishes.
Moreover, $0\leq s\mu+1-s \leq 1$, whence we get the inequality:
\[ 
 \FldA\bigl( (s\mu+1-s)\cdot\Delta(\dif) \bigr) =(s\mu+1-s) \FldA\bigl(\Delta(\dif)\bigr) > -1,
\]
as well.
Thus $H(\dif,s)$ is a diffeomorphism.

4) Since $\func\circ\FlowA(z,\tau,t)=\func(z,\tau)$, it follows that $\func\circ H(\dif,s)=\func$ for all $(\dif,s)\in\Stab_{J}\times I$.
Thus $H(\dif,s)\in \Stab_{J}$.

5) Let us show that $H_1(\Stab_J)\subset \Stab_{T}$, i.e. $H(\dif,1)$ is fixed on $S^1\times T$.
Let $\tau\in T$.
Then $\mu(\tau)=0$, whence 
\[
 H(\dif, 1)(z,\tau) = \FlowA\bigr(z,\tau, (1\cdot \mu(\tau)+1-1)\cdot\Delta(\dif)(z,\tau)\bigr) = 
\FlowA(z,\tau, 0)=(z,\tau).
\]

6) Finally, let us verify that $H(\dif,s)=\dif$ on $S^1\times(I\setminus T')$.
Let $\tau\in I\setminus T'$.
Then $\mu(\tau)=1$, whence
\[
 H(\dif, s)(z,\tau) = \FlowA\bigr(z,\tau, (s \mu(\tau)+1-s)\cdot\Delta(\dif)(z,\tau)\bigr) = 
\FlowA\bigr(z,\tau, \Delta(\dif)(z,\tau)\bigr)=\dif(z,\tau).
\]
Lemma is proved.
\end{proof}

\begin{remark}
Notice that the map $H_1:\Stab_J\to\Stab_T$ is not a retraction.
\end{remark}

Let $\Xman$ be an $\func$-adopted submanifold with $X^0=\varnothing$ and $\Nman$ be a neighbourhood of $\Xman$.
For every connected component $\Yman$ of $\Xman$ let $\Nman_{\Yman}$ be the connected component of $\Nman$ containing $\Yman$.
\begin{definition}
Say that $\Nman$ is \myemph{$\func$-adopted} if it has the following properties.
\begin{enumerate}
 \item
$\overline{\Nman_{\Yman}}\cap \overline{\Nman_{\Yman'}}=\varnothing$ for any pair of distinct components $\Yman, \Yman'$ of $\Xman$.
 \item 
Let $\Yman$ be a connected component of $\Xman^1$.
Put $J=[0,1]$ if $\Yman$ is a boundary component of $\Mman$, and $J=[-1,1]$ otherwise.
Then there exists a diffeomorphism $q:S^1\times J \to \Nman_{\Yman}$ such that $q(S^1\times 0)=\Yman$, for each $t\in J$ the set $q(S^1\times t)$ is a regular component some level set of $\func$.
 \item
Let $\Yman$ be a connected component of $\Xman^2$, and $\gamma_1,\ldots,\gamma_n$ be the set of all boundary components of $\partial\Yman$ that belong to the interior $\Int{\Mman}$.
Then $\Nman_{\Yman}$ is obtained from $\Yman$ by gluing collars $C_i=S^1\times[0,1]$ to each of $\gamma_i$ along $S^1\times 0$, so that for every $t\in [0,1]$ the set $S^1\times t$ corresponds to some level set of $\func$.
\end{enumerate}
\end{definition}

As a consequence of Lemma~\ref{lm:deform_near_X} we get the following statement.
\begin{corollary}\label{cor:DiffMX'_DiffMX}
Let $\Xman\subset\Mman$ be an $\func$-adopted submanifold, and $\hat\Nman$ be an $\func$-adopted neighbourhood of $\Xman^1\cup \Xman^2$.
Denote $\hat\Xman = \Xman^0 \cup \hat\Nman$.
Then the inclusion \[\Stab(\func,\hat\Xman)\cap\Diff(\Mman,\hat\Xman) \ \subset \ \StabfX\cap\DiffMX\] is a homotopy equivalence.
In particular, so is the inclusion $\Stab'(\func,\hat\Xman) \subset \Stab'(\func,\Xman)$.
\end{corollary}
\begin{proof}
Let $\dif \in \StabfX\cap\DiffMX$.
We should construct a ``canonical'' deformation of $\dif$ in $\StabfX$ to a diffeomorphism fixed on a neighbourhood $\hat\Nman$ on $\Xman^1\cup \Xman^2$, and this deformation should be supported in some neighbourhood of $\Xman^1\cup\partial\Xman^2$.

Let $\Yman$ be a connected component of $\Xman^1\cup\partial\Xman^2$.
Then $\Yman$ has a neighbourhood $\Uman$ diffeomorphic to the cylinder $S^1\times I$ such that each set $S^1\times \tau$ is a regular component of some level set of $\func$.
Then by Lemma~\ref{lm:deform_near_X} there exists a deformation of $\Stab(\func|_{\Uman}, \Uman\cap\Xman)$ into $\Stab(\func|_{\Uman}, \Uman\cap\hat\Nman)$ with supports in $\Int{\Uman}$.

Applying this to each connected component of $\Xman^1\cup\partial\Xman^2$ we will get a deformation retraction of $\StabfX\cap\DiffMX$ onto $\Stab(\func,\hat\Xman)\cap\Diff(\Mman,\hat\Xman)$.
The details are left to the reader.
\end{proof}

\begin{corollary}\label{cor:decomp_Of}
Suppose $\Xman^1\cup\Xman^2\not=\varnothing$.
Let $\Mman_1,\ldots,\Mman_n$ be the closures of the connected components of $\Mman\setminus(\Xman^1\cup\Xman^2)$, and $\Yman_i = \Mman_i \cap \Xman$.
Then we have the following commutative diagram consisting of isomorphisms:
\begin{equation}\label{equ:CD:iso_p1Of_pi0S}
\begin{CD}
 \pi_1\OrbffX @>{\mu}>> \prod_{i=1}^{n} \pi_1\Orbit(\func|_{\Mman_i},\Yman_i) \\
@V{\partial_1}VV  @VV{\prod_{i=1}^{n}(\partial_1)_i}V \\
 \pi_0\Stab'(\func,\Xman) @>{\eta}>>
\prod_{i=1}^{n} \pi_0\Stab'(\func|_{\Mman_i}, \Yman_i).
\end{CD}
\end{equation}
for some isomorphisms $\mu$ and $\eta$, where $(\partial_1)_i:\pi_1\Orbit(\func|_{\Mman_i},\Yman_i) \to\pi_0\Stab'(\func|_{\Mman_i}, \Yman_i)$ is the boundary homomorphism.
\end{corollary}
\begin{proof}
Since $\Xman$ and each $\Yman_i$ is infinite, it follows from~\eqref{equ:pi1Of_case5} that $\partial_1$ and $(\partial_1)_i$ are isomorphisms, and so are the vertical arrows.
It is sufficient to define $\eta$, then $\mu$ will be uniquely determined.

Let $\hat\Nman$ be an $\func$-adopted neighbourhood of $\Xman^1\cup\Xman^2$ and $\hat\Xman=\Xman^0\cup\hat\Nman$.
Denote $\hat\Yman_i = \Mman_i \cap \hat\Xman$, $i=1,\ldots,n$.
Then by Corollary~\ref{cor:DiffMX'_DiffMX} we have isomorphisms
$\pi_0\Stab'(\func,\Xman) \approx \pi_0\Stab'(\func,\hat\Xman)$, and 
$\pi_0\Stab'(\func|_{\Mman_i}, \Yman_i) \approx \pi_0\Stab'(\func|_{\Mman_i}, \hat\Yman_i)$.

Notice that the following map
\[
 \eta': \Stab'(\func,\hat\Xman) \longrightarrow \prod_{i=1}^{n} \Stab'(\func|_{\Mman_i},\hat\Yman_i),
\qquad 
\eta'(\dif) = \bigl(\dif|_{\Mman_1}, \ldots, \dif|_{\Mman_n} \bigr)
\]
is a group isomorphism, since the restrictions $\dif|_{\Mman_i}$, $i=1,\ldots,n$, have disjoint supports.
Therefore $\eta'$ it induces an isomorphism $\eta$ from~\eqref{equ:CD:iso_p1Of_pi0S}.
\end{proof}

\subsection*{Proof of Theorem~\ref{th:compute_pi1Of}}
Consider the following cases.
{(a)}~{\em $\chi(\Mman)<0$ and $\Xman$ is finite.}
For $\Xman=\varnothing$ the result was proved in~\cite[Th.~1.8]{Maks:MFAT:2010}.
However, the analysis of the proof shows that the same arguments hold for $\Xman\subset\singf$ as well.

{(b)}~{\em $\chi(\Mman)<0$, $\varnothing \not= \partial\Mman \subset \Xman \subset \partial\Mman \cup \singf$.}
By Corollary~\ref{equ:OffX_OffY} we have that $\OrbffX=\Orbit(\func,\Xman^0)$, whence the decomposition from (a) holds in this case.

{(c)}~{\em $\Xman$ is infinite.}
Again by Corollary~\ref{equ:OffX_OffY} we can assume that $\partial\Mman\subset\Xman$.
Then by Corollary~\ref{cor:decomp_Of} we can write
\[
 \pi_1\OrbffX \approx \prod_{i=1}^{n} \pi_1\Orbit(\func|_{\Mman_i},\Yman_i),
\]
where $\Mman_1,\ldots,\Mman_n$ are the closures of the connected components of $\Mman\setminus(\Xman^1\cup\Xman^2)$, and $\Yman_i=\partial\Mman_i\cup(\Mman_i\cap\Xman^0)\not=\varnothing$.
If $\chi(\Mman_i)\geq0$, then $\Mman_i$ is either a $D^2$, $S^1\times I$, or $\MobiusBand$, and so it satisfies the statement of theorem.
Otherwise, $\chi(\Mman_i)<0$, and we can decompose $\pi_1\Orbit(\func|_{\Mman_i},\Yman_i)$ by the case (b).

Evidently, (a)-(c) include all the cases (i)-(iii). 
Theorem is completed.

\section{Axioms for a map $\func:\Mman\to\Pman$}
In this section we will present additional three axioms \AxVF-\AxRestr\ for a smooth map $\func:\Mman\to\Pman$ satisfying axiom \AxBd.
These axioms are consequences of \AxCrPt.
In the last two sections we will prove Theorems~\ref{th:StabIdf} and \ref{th:DiffMX_to_OrbfX_fibr} for maps $\func:\Mman\to\Pman$ satisfying \AxBd-\AxRestr.

First we introduce some notation.
For a vector field $\FldA$ on $\Mman$ tangent to $\partial\Mman$ denote by $\FlowA:\Mman\times\RRR\to\Mman$ the flow of $\FldA$, and by $\ShA:\Ci{\Mman}{\Pman}$ the \myemph{shift map} of $\FldA$ defined by \[\ShA(\afunc)(x)=\FlowA(x,\afunc(x))\] for $\afunc\in\Ci{\Mman}{\Pman}$ and $x\in\Mman$.

Say that a vector-field $\FldA$ on $\Mman$ is \emph{skew-gradient} with respect to $\func$, if $\FldA(\func)\equiv 0$,  and $\FldA(z)=0$ if and only if $z$ is a critical point of $\func$.
In particular, $\func$ is constant along orbits of $\FldA$.

Let $\Mman$ be a non-orientable compact surface. 
Then we will always denote by $\DoubleCover:\tMman\to\Mman$ the orientable double covering of $\Mman$ and by $\Invol$ the orientation reversing involution of $\tMman$ which generates the group $\ZZZ_2$ of covering transformations of $\tMman$. 

Moreover, for a $\Cinf$-map $\func:\Mman\to\Pman$ we put $\tfunc=\DoubleCover\circ\func:\tMman\to\Pman$, and denote by $\tDifftM$ the group of diffeomorphisms $\dif$ of $\tMman$ commuting with $\Invol$, i.e. $\tdif\circ\Invol=\Invol\circ\tdif$.
Let also $\tStabtf=\{\tdif\in\tDifftM \mid \tfunc\circ\tdif=\tfunc\}$ be the stabilizer of $\tfunc$ with respect to the right action of the group $\tDifftM$, and $\tStabIdtf$ be the identity path component of $\tStabtf$. 

Notice that each $\tdif\in\tDiffIdtM$ induces a unique diffeomorphism $\dif$ of $\Mman$, and the correspondence $\tdif\mapsto\dif$ is a homeomorphism $\nu:\tDiffIdtM \to \DiffIdM$, which induces a homeomorphism $\nu:\tStabIdtf\to\StabIdf$.

\begin{axiom}{\AxVF}
Suppose $\Mman$ is orientable.
Then there exists a \myemph{skew-gradient} with respect to $\func$ vector field $\FldA$ on $\Mman$ satisfying the following conditions.

\begin{itemize}
\item
Consider the following \myemph{convex} subset of $\Ci{\Mman}{\RRR}$:
\[\Gamma =\{\afunc\in\Ci{\Mman}{\RRR} \ | \ \FldA(\afunc)>-1\}.\]
Then $\ShA(\Gamma)=\StabIdf$.
\item
If $\func$ has a critical point which is either non-extremal or degenerate extremal, then $\ShA|_{\Gamma}:\Gamma\to\StabIdf$ is a homeomorphism with respect to $\Cinf$-topologies, and so $\StabIdf$ is contractible.

Otherwise, $\ShA|_{\Gamma}:\Gamma\to\StabIdf$ is a $\ZZZ$-covering map, and $\StabIdf\cong S^1$.
In this case there is a \myemph{strictly positive} function $\theta\in\Gamma$ such that for any $\afunc,\bfunc\in\Gamma$ we have that $\ShA(\afunc)=\ShA(\bfunc)$ if and only if $\afunc-\bfunc = n\theta$ for some $n\in\ZZZ$.
\end{itemize}

If $\Mman$ is non-orientable, then there exists a \myemph{skew-gradient} with respect to $\tfunc$ vector field $\FldA$ on $\tMman$ satisfying the following conditions.
\begin{itemize}
\item 
$\FldA$ is skew-symmetric with respect to $\Invol$, in the sense that $\Invol^{*}\FldA = -\FldA$, which is equivalent to the assumption that $\FlowA_{\theta}\circ\Invol=\Invol\circ\FlowA_{-\theta}$ for all $\theta\in\RRR$.

\item
Define the following \myemph{convex} subset $\Ci{\tMman}{\RRR}$:
\[\widetilde\Gamma =\{\afunc\in\Ci{\tMman}{\RRR} \ | \ \FldA(\afunc)>-1, \ \afunc\circ\Invol=-\afunc\}.\]
Then $\ShA(\widetilde\Gamma)=\tStabIdtf$, and the restriction of shift map $\ShA:\widetilde\Gamma\to\tStabIdtf$ is a homeomorphism with respect to $\Cinf$-topologies, whence $\tStabIdtf$ and $\StabIdf=\nu\bigl(\tStabIdtf\bigr)$ are contractible.
\end{itemize}
\end{axiom}

\begin{axiom}{\AxFibr}
The map $p:\DiffM\to\Orbf$ defined by $p(\dif) = \func\circ\dif$ for $\dif\in\DiffM$ is a Serre fibration.
\end{axiom}

\begin{axiom}{\AxRestr}
Let $\Yman \subset \Mman$ be a subsurface such that $\func|_{\Yman}$ satisfies \AxBd.
If $\func$ also satisfies \AxVF\ and \AxFibr, then so does $\func|_{\Yman}$.
\end{axiom}

The following lemma summarizes certain results obtained in~\cite{Maks:AGAG:2006, Maksymenko:ProcIM:ENG:2010}.

\begin{lemma}\label{lm:SuffCondAx}\cite{Maksymenko:ProcIM:ENG:2010}
Axioms \AxBd\ and \AxCrPt\ imply all other axioms \AxVF-\AxRestr.
\end{lemma}
\begin{proof}
In the paper~\cite{Maksymenko:ProcIM:ENG:2010} the author introduced three axioms {\rm(A1)-(A3)} for a smooth map $\func:\Mman\to\Pman$ such that \AxBd=(A1), \AxFibr=(A3), 
\AxBd\&\AxCrPt$\Rightarrow$(A1)-(A3) by \cite[Lm.~12]{Maksymenko:ProcIM:ENG:2010}, and (A1)-(A3)$\Rightarrow$\AxVF\ by \cite[Th.~3]{Maksymenko:ProcIM:ENG:2010}.
To verify \AxRestr, suppose $\Yman \subset \Mman$ is a submanifold such that $\func|_{\Yman}$ satisfies \AxBd.
Then $\func|_{\Yman}$ satisfies \AxCrPt, and therefore all other axioms hold true.
\end{proof}

\section{Proof of Theorem~\ref{th:StabIdf}}\label{sect:proof_th_1}
The orientable case of Theorem~\ref{th:StabIdf} is contained in the following lemma:
\begin{lemma}\label{lm:hom_type_StabIdfX}
Suppose $\Mman$ is orientable, and $\func:\Mman\to\Pman$ satisfies \AxBd\ and \AxVF.
Let $\FldA$, $\FlowA$, $\ShA$, and $\Gamma$ be the same as in \AxVF.
Denote 
$\Gamma_{\Xman} = \{\afunc\in\Gamma \ \mid \ \alpha|_{\Xman^1\cup\Xman^2} = 0\}.$
Then 
\begin{equation}\label{equ:ShAGX_StabfIdfX}
\ShA(\Gamma_{\Xman}) = \StabIdfX.
\end{equation}
Moreover, $\StabIdfX\cong S^1$ iff $\StabIdf\cong S^1$ and $\Xman\subset\singf$.
Otherwise $\StabIdfX$ is contractible.
\end{lemma}
\begin{proof}
\eqref{equ:ShAGX_StabfIdfX}.
Let $\afunc\in\Gamma_{\Xman}$, i.e. $\afunc(x)=0$ for all $x\in\Xman^1\cup\Xman^2$.
Then $\ShA(\afunc)$ is fixed on $\Xman$, i.e. $\ShA(\afunc)\in\StabIdfX$.

Indeed, if $x\in\Xman^1\cup\Xman^2$, then $\ShA(\afunc)(x) = \FlowA(x,\afunc(x))=\FlowA(x,0)=x$. 
Moreover, every $x\in\Xman^0$ is a critical point of $\func$, and so $\FlowA(x,t)=x$ for all $t\in\RRR$.
In particular, $\ShA(\afunc)(x) = \FlowA(x,\afunc(x))=x$, and so $\ShA(\afunc)\in\StabfX$.

Since $\Gamma$ is connected, $\ShA(0)=\id_{\Mman}\in\StabfX$, and $\ShA$ is continuous, we get that $\ShA(\afunc)\in\StabIdfX$.

Conversely, suppose $\dif\in\StabIdfX$, so we have an isotopy $\dif_t:\Mman\to\Mman$ in $\StabIdfX$ between $\dif_0=\id_{\Mman}$ and $\dif_1=\dif$.
Since $\ShA$ induces a covering map of $\Gamma$ onto $\StabIdf$, we can lift the homotopy $\dif_t$ (regarded as a continuous path in $\StabIdfX\subset\StabIdf$) to $\Gamma$ and get a homotopy of functions $\afunc_t:\Mman\to\RRR$, $t\in[0,1]$, such that $\dif_t(x)=\FlowA(x,\afunc_t(x))$.
Moreover, as $\dif_0=\id_{\Mman}$ we can assume that $\afunc_0\equiv0$.
We claim that $\afunc_t\in\Gamma_{\Xman}$ i.e. $\afunc_t=0$ on $\Xman^1\cup\Xman^2$.

For each point $x\in\Xman^1\cup\Xman^2$ consider the set $\Lambda_x = \{ \tau\in\RRR \mid \FlowA(x,\tau)=x\}$ of periods of $x$, so $\afunc_t(x)\in\Lambda_x$ for all $t\in[0,1]$.
Notice that $\Lambda_x=\{0\}$  if $x$ is non-periodic point, $\Lambda_x=\theta_x\ZZZ$ for a periodic point of period $\theta_x$, and $\Lambda_x=\RRR$ if $x$ is  a fixed point, i.e. a critical point of $\func$.

Since for every non-fixed point $x$ the set $\Lambda_x$ is discrete and contains $0$, and $\afunc_0=0$, it follows that $\afunc_t=0$ on $(\Xman^1\cup\Xman^2)\setminus\singf$ for all $t\in\RRR$.
But $\singf$ is nowhere dense, therefore $\afunc_t=0$ on all of $\Xman^1\cup\Xman^2$, i.e. $\afunc_t\in\Gamma_{\Xman}$.
This proves~\eqref{equ:ShAGX_StabfIdfX}.

Now we can describe the homotopy type of $\StabIdfX$.
Notice that $\Gamma$ and $\Gamma_{\Xman}$ are convex subsets of $\Ci{\Mman}{\RRR}$ and therefore contractible.

1) If $\StabIdfX$ is contractible, i.e. $\ShA:\Gamma\to\StabIdf$ is a homeomorphism, then $\ShA:\Gamma_{\Xman}\to\StabIdfX$ is a homeomorphism as well, whence $\StabIdfX$ is also contractible.

2) Suppose $\StabIdfX\cong S^1$, so $\ShA:\Gamma\to\StabIdf$ is a $\ZZZ$-covering map.
Consider two cases.

a) Suppose $\Xman^1\cup\Xman^2=\varnothing$, and so $\Xman=\Xman^0\subset\singf$.
In this case $\Gamma_{\Xman}=\Gamma$, whence $\StabIdfX=\StabIdf \cong S^1$.

b) Let $\Xman^1\cup\Xman^2\not=\varnothing$.
Then the restriction map $\ShA|_{\Gamma_{\Xman}}:\Gamma_{\Xman}\to\StabIdfX$ is injective.
Hence due to~\eqref{equ:ShAGX_StabfIdfX} it is a homeomorphism onto.

Indeed, suppose $\ShA(\afunc)=\ShA(\bfunc)$ for some $\afunc,\bfunc\in\Gamma_{\Xman}$.
Then by \AxVF, $\afunc = \bfunc + n\theta$ for some $n\in\ZZZ$.
However, $\theta>0$ on all $\Mman$, while $\afunc=\bfunc=0$ on $\Xman^1\cup\Xman^2$.
Therefore $n=0$ and so $\afunc=\bfunc$ of all of $\Mman$.
\end{proof}

Suppose $\Mman$ is non-orientable.
Then $\tXman = \DoubleCover^{-1}(\Xman)$ is $\tfunc$-adopted submanifold of $\tMman$.
Let $\tStab(\tfunc,\tXman)$ be the subgroup of $\Stab(\tfunc,\tXman)$ consisting of diffeomorphisms commuting with $\Invol$ and $\tStabIdtftX$ be its identity path-component.
By the arguments similar to the proof of Lemma~\ref{lm:hom_type_StabIdfX} one can show that $\tStabIdtftX$ is contractible.
Moreover, the homeomorphism $\nu:\tStabIdtf\to\StabIdf$ maps $\tStabIdtftX$ onto $\StabIdfX$, whence 
$\StabIdfX$ is contractible as well.

\section{Proof of Theorem~\ref{th:DiffMX_to_OrbfX_fibr}}
\label{sect:proof_th_2}

Theorem~\ref{th:DiffMX_to_OrbfX_fibr} is contained in the following theorem. 
\begin{theorem}\label{th:p_DiffMX_to_OrbfX_is_Serre}
Suppose $\func:\Mman\to\Pman$ satisfies axioms \AxBd-\AxRestr.
Then the restriction map $p|_{\DiffMX}:\DiffMX\to\OrbfX$ is a Serre fibration as well.
\end{theorem}
\begin{proof}
Let $\Sman$ be a finite path-connected CW-complex and $s_0\in \Sman$ be a point.
Let also $\InitHom:\Sman\times I\to \OrbfX$ be a homotopy such that $\InitHom(s_0,0)=\func$, and the restriction $\InitHom|_{\Sman\times0}:\Sman\times 0\to \OrbfX$ lifts to a map $\ReqLift_0:\Sman\to\DiffMX$ satisfying $\ReqLift_0(s_0)=\id_{\Mman}$ and $\InitHom(s,0)=p(\ReqLift_0(s)) = \func\circ\ReqLift_0(s)$.
In particular, $\ReqLift_0(s)|_{\Xman}=\id_{\Xman}$ for all $s\in \Sman$.
 
We will prove that $\ReqLift_0$ extends to a map $\ReqLift:\Sman\times I\to \DiffMX$ such that $\InitHom=p\circ\ReqLift$.

Since $p:\DiffM\to\Orbf$ is a Serre fibration, it follows that $\ReqLift_0$ extends to a map $\InitLift:\Sman\times I\to\DiffM$ such that $\InitHom=p\circ\InitLift$.
So we have the following commutative diagram:
\[
 \xymatrix{
\Sman\times 0 \ar@{^{(}->}[d] \ar[r]^{~\ReqLift_0\quad} & \DiffMX  \ar@{^{(}->}[r] & \DiffM \ar[d]^{p} \\
\Sman\times I \ar[r]_{~\InitHom\quad} \ar[rru]^{\InitLift} \ar@{-->}[ru]^{\ReqLift} & \OrbfX \ar@{^{(}->}[r] & \Orbf \\
}
\]

Notice also that $\InitLift$ induces a continuous map
\[
 \DefInitLift:\Sman\times I \times \Mman\to\Mman, 
\qquad 
 \DefInitLift(s,t,x) =\InitLift(s,t)(x).
\]

\begin{lemma}\label{lm:X_is_invar}
Let $(s,t)\in \Sman\times I$ and $\gamma$ be a leaf of $\partitf$ contained in $\Xman$.
Then $\InitLift(s,t)(\gamma)=\gamma$.
If $\gamma$ is $1$-dimensional, then $\InitLift(s,t)$ also preserves orientation of $\gamma$.
\end{lemma}
\begin{proof}
By definition $\InitHom(s,t) = \func\circ\InitLift(s,t)$.
On the other hand, as $\InitHom(s,t)\in \OrbfX$, there exists $\XInitLift\in\DiffMX$ which depends on $(s,t)$ and such that $\InitHom(s,t)=\func\circ\XInitLift$ as well.
Hence $ \func\circ\InitLift(s,t) \circ \XInitLift^{-1} = \func$, and so
\begin{align*}
\InitLift(s,t) \circ \XInitLift^{-1}(\singf)  &=  \singf, &
\InitLift(s,t) \circ \XInitLift^{-1}(\func^{-1}(c))  &=  \func^{-1}(c)
\end{align*}
for all $c\in\Pman$.
Moreover, as $\XInitLift$ is fixed on $\Xman$, we obtain that
\begin{align*}
\InitLift(s,t) \bigl(\singf\cap\Xman\bigr)  &\subset\singf, &
\InitLift(s,t) \bigl(\func^{-1}(c)\cap\Xman\bigr)  &\subset  \func^{-1}(c),
\end{align*}
for all $(s,t)\in\Sman\times I$, that is
\begin{align*}
\DefInitLift\Bigl(\Sman\times I \times [\singf\cap\Xman\bigr] \Bigr)  &\subset \singf, & \ \ \
\DefInitLift\Bigl(\Sman\times I \times [\func^{-1}(c)\cap\Xman] \Bigr) &\subset \func^{-1}(c).
\end{align*}

Now notice that there are three possibilities for $\gamma$:
\begin{itemize}
 \item[(i)]
$\gamma$ is a critical point of $\func$,
 \item[(ii)]
$\gamma$ is a regular component of some level set $\func^{-1}(c)$, $c\in \Pman$,
 \item[(iii)]
there is a critical component $\omega$ of some level set $\func^{-1}(c)$ such that $\gamma$ is a connected component of $\omega\setminus\singf$.
\end{itemize}

Let $z$ be a critical point of $\func$ and $\omega$ be a connected component of $\func^{-1}(c)\cap\Xman$.
Since $\InitLift(s_0,0)=\id_{\Mman}$ and the sets $\Sman\times I \times \{z\}$ and $\Sman\times I \times \omega$ are connected, we obtain that 
$\DefInitLift\bigl(\Sman\times I \times \{z\}\bigr)= \{z\}$,  and
$\DefInitLift\bigl(\Sman\times I \times \omega \bigr) = \omega$.
Hence $\DefInitLift\bigl(\Sman\times I \times \alpha \bigr) = \alpha$ for any connected component $\alpha$ of $\omega\setminus\singf$.
In other words,
\begin{align*}
\mathrm{(i)}~\InitLift(s,t)(z) &= z, &
\mathrm{(ii)}~\InitLift(s,t)(\omega) &= \omega, &
\mathrm{(iii)}~\InitLift(s,t)(\alpha) &= \alpha,
\end{align*}
and, moreover, $\InitLift(s,t)$ preserves orientation of $\omega$ and $\alpha$ because $\InitLift(s_0,0)$ does so.
This proves our lemma for all the cases (i)-(iii) of $\gamma$.
\end{proof}
The lemma says, in particular, that the lifting $\InitLift$ fixes $\Xman^0$.
We will find the lifting which also fixes $\Xman^1 \cup \Xman^2$.

Choose an $\func$-adopted neighbourhood of $\Xman^1\cup\Xman^2$.
Let also $\Yman$ be a connected component of $\Xman^1\cup \Xman^2$.
We will now distinguish three cases of $\Yman$:
\begin{enumerate}
 \item[\rm(A)]
$\Yman$ is an orientable surface, 
 \item[\rm(B)]
$\Yman$ is a non-orientable surface, 
 \item[\rm(C)]
$\Yman$ is a regular component of some level set of $\func$, so $\Yman$ is a circle.
\end{enumerate}

\begin{lemma}\label{lm:exists_sh_func}
Suppose $\Yman$ belongs to the cases {\rm(A)} or {\rm(C)}.
Let also $\FldA$, $\FlowA$, $\ShA$, and $\Gamma$ be the same as in Axiom~\AxVF.
Then there exists a continuous map $\hat\ShFInitLift:\Sman\times I \to \Gamma\subset \Ci{\Mman}{\RRR}$ having the following properties:
\begin{enumerate}
 \item[\rm(a)]
$\hat\ShFInitLift(s,0)=0$ of $\Mman$ for all $s\in \Sman$;

 \item[\rm(b)]
$\supp\bigl(\hat\ShFInitLift(s,t)\bigr) \subset \Nman_{\Yman}$ for each $(s,t)\in \Sman\times I$;

 \item[\rm(c)]
Define the map 
$\hat\InitLift = 
\ShA \circ \hat\ShFInitLift:\Sman\times I \xrightarrow{~\hat\ShFInitLift~} \Gamma \xrightarrow{~\ShA~} \StabIdf$,
so
\[
 \hat\InitLift(s,t)(x) = \FlowA\bigl(x,\hat\ShFInitLift(s,t)(x)\bigr).
\]
Then $\hat\InitLift(s,t)=\InitLift(s,t)$ on $\Yman$ for each $(s,t)\in \Sman\times I$.
\end{enumerate}
Hence the map $\ReqLift:\Sman\times I \to \DiffMX$ defined by
\[
 \ReqLift(s,t) = \hat\InitLift(s,t)^{-1} \circ \InitLift(s,t)
\]
is a required lifting of $\InitHom$.
\end{lemma}
\begin{proof}
{\bf Case (A).} Now $\Yman$ is an \emph{orientable} component of $\Xman^2$.
Then by \AxRestr\ the restriction of $\func$ to $\Yman$ satisfies axiom \AxVF.
Let
\[
 \Gamma_{\Yman} = \{ \afunc\in\Ci{\Yman}{\RRR} \mid \FldA(\afunc)>-1 \}
\]
and $\ShA_{\Yman}:\Gamma_{\Yman}\to\StabId(\func|_{\Yman})$, $\ShA_{\Yman}(\afunc)(x) =\FlowA(x,\afunc(x))$, be the corresponding covering map described in (2) of \AxVF.

Due to Lemma~\ref{lm:X_is_invar} $\func\circ\InitLift(s,t)(x) = \func(x)$ for $(s,t,x)\in \Sman\times I\times \Yman$.
In other words, the restriction $\InitLift(s,t)|_{\Yman}$ of $\InitLift(s,t)$ to $\Yman$ belongs to the stabilizer $\Stab(\func|_{\Yman})$ of the restriction $\func|_{\Yman}$ with respect to the right action of the group $\Diff(\Yman)$.
Moreover, since $\InitLift(s,0)=\id_{\Mman}$ and $\Sman\times I\times\Yman$ is connected, it follows that $\InitLift(s,t)|_{\Yman}\in\StabId(\func|_{\Yman})$.

Consider the restriction to $\Yman$ map:
\[
 \InitLift_{\Yman}:\Sman\times I \to \StabId(\func|_{\Yman}),
\qquad 
 \InitLift_{\Yman}(s,t)(x) = \InitLift(s,t)(x)
\]
for $(s,t,x)\in \Sman\times I\times \Mman$.
This map is continuous into $\Cinf$-topology of $\StabId(\func|_{\Yman})$.

Since $\ShA_{\Yman}|_{\Gamma_{\Yman}}$ is a covering map, and $\InitLift_{\Yman}(s,0)=\id_{\Yman}$ for all $s\in \Sman$, we can lift $\InitLift_{\Yman}|_{\Sman\times 0}$ to a map $\ShFInitLift:\Sman\times 0\to \Gamma_{\Yman}$ by the formula $\ShFInitLift(s,0)=0:\Yman\to\RRR$ for all $s\in\Sman$.
Then from covering homotopy property of $\ShA_{\Yman}|_{\Gamma_{\Yman}}$ we get that $\InitLift_{\Yman}$ extends to a lift $\ShFInitLift:\Sman\times I\to \Gamma_{\Yman}$ such that the following diagram is commutative:
\begin{equation}\label{equ:lift_to_Gamma}
 \xymatrix{
&  \Gamma_{\Yman} \ar[d]^{\ShA_{\Yman}}  \ar@{^{(}->}[r] & \Ci{\Yman}{\RRR} \\
\Sman\times I \ar[ru]^{\ShFInitLift} \ar[r]^{\InitLift_{\Yman}} & \StabId(\func|_{\Yman})  \ar@{^{(}->}[r]  &  \Diff(\Yman)
}
\end{equation}
In other words,
\begin{equation}\label{equ:InitLift_is_a_shift}
 \InitLift(s,t)(x)=\InitLift_{\Yman}(s,t)(x) = \FlowA\bigl(x,\ShFInitLift(s,t)(x)\bigr)
\end{equation}
for $x\in\Yman$.

Notice that for a neighbourhood $\Nman_{\Yman}$ there exists an \emph{linear} extension operator 
\[
 E:\Ci{\Yman}{\RRR} \longrightarrow \Ci{\Mman}{\RRR} 
\]
such that $E\afunc|_{\Yman} = \afunc$, and $\supp(E\afunc)\subset \Nman_{\Yman}$ for each $\afunc\in\Ci{\Yman}{\RRR}$, see~\cite{Seeley:PAMS:1964}.
Define the composition 
\[
 \ShFInitLift' = E \circ \ShFInitLift:\Sman\times I \longrightarrow\Ci{\Mman}{\RRR}
\]
and consider the following map 
\[
 \DefInitLift':\Sman\times I \times \Mman\to \Mman,
\qquad 
 \DefInitLift'(s,t,x) =  \FlowA\bigl(x,\ShFInitLift'(s,t)(x)\bigr).
\]
Evidently, $\DefInitLift'(s,t,x)=\DefInitLift(s,t,x)$ for $x\in\Yman$, so the restriction $\DefInitLift_{s,t}$ to $\Yman$ is a diffeomorphism, and therefore we have the following inequality, see~\eqref{equ:Shift_is_diff}: 
\begin{equation}\label{equ:Fhd_m1}
 \FldA(\ShFInitLift'(s,t))(x)> -1, \qquad x\in\Yman.
\end{equation}
As $\Sman\times I \times \Yman$ is compact and partial derivatives of $\ShFInitLift'(s,t)$ are continuous in $(s,t,x)$, it follows that~\eqref{equ:Fhd_m1} holds for all $x$ belonging to some neighbourhood of $\Yman$ which does not depend on $(s,t)$.
Decreasing $\Nman_{\Yman}$ we can assume that~\eqref{equ:Fhd_m1} holds on all of $\Nman_{\Yman}$.

Let $\Wman$ be a neighbourhood of $\Yman$ such that
\begin{equation}\label{equ:nbh_of_Y}
\Yman \subset \Wman \subset \overline{\Wman} \subset \Nman_{\Yman}.
\end{equation}

Take a $\Cinf$ function $\mu:\Mman\to[0,1]$ such that (i)~$\mu=1$ on $\Yman$; (ii)~$\mu=0$ on $\Mman\setminus\Wman$; (iii)~$\mu$ takes constant values on connected components of level sets of $\func$, so $\FldA(\mu)=0$.

Now define the map $\hat\ShFInitLift:\Sman\times I \to \Ci{\Mman}{\RRR}$ by
\[
 \hat\ShFInitLift(s,t)(x)= \mu(x)\ShFInitLift'(s,t)(x).
\]
Notice that 
\[
 \FldA\bigl(\hat\ShFInitLift(s,t)\bigr)=
\FldA\bigl(\mu\cdot\ShFInitLift'(s,t)\bigr) = 
 \mu \FldA(\ShFInitLift') +  \FldA(\mu)\ShFInitLift'=
 \mu \FldA(\ShFInitLift') > -1.
\] 
The latter inequality follows from~\eqref{equ:Fhd_m1} and the assumption that $0\leq\mu\leq 1$.
Hence $\hat\ShFInitLift(\Sman\times I)\subset\Gamma$.

We have to show that $\hat\ShFInitLift$ satisfies conditions (a)-(c) of lemma.

(a) Since $\ShFInitLift(s,0)=0$ on $\Yman$ and $E$ is a linear operator, it follows that $\ShFInitLift'(s,0)=0$ on $\Mman$, and therefore $\hat\ShFInitLift(s,0)=0$ on $\Mman$ as well.

(b) Since $\supp(\mu)\subset\Nman_{\Yman}$, we have that $\supp(\hat\ShFInitLift(s,t))\subset\Nman_{\Yman}$.

(c) Define the map $\hat\InitLift:\Sman\times I\to \Ci{\Mman}{\Mman}$ by 
\[
 \hat\InitLift(s,t)(x) = \FlowA\bigl(x,\hat\ShFInitLift(s,t)(x)\bigr).
\]
Then it follows from~\eqref{equ:InitLift_is_a_shift} that $\hat\InitLift(s,t)=\InitLift(s,t)$ on $\Yman$.
Therefore $\ReqLift(s,t) = \hat\InitLift(s,t)^{-1} \circ \InitLift(s,t)$ is fixed on $\Yman$.
Moreover, since $\hat\InitLift(s,t)$ preserves leaves of $\partitf$, we see that $\func\circ\hat\InitLift(s,t)^{-1}=\func$, and therefore
\[
 \func\circ\ReqLift(s,t)= \func\circ\hat\InitLift(s,t)^{-1} \circ \InitLift(s,t)=
\func\circ \InitLift(s,t) = \InitHom(s,t).
\]

\medskip

{\bf Case (C)}. Suppose $\Yman$ is a regular component of some level set of $\func$, so we can identify $\Nman_{\Yman}$ with the product $S^1\times J$, where $S^1=\{z\in\CCC \mid |z|=1\}$ is the unit circle in the complex plane, and $J=[-1,1]$ or $[0,1]$, so that $\Yman$ corresponds to $S^1\times 0$.
The proof in this case is similar to the proof of Lemma~\ref{lm:deform_near_X}.

Define a flow $\FlowA:(S^1\times J)\times \RRR\to S^1\times J$ by $\FlowA(z,\tau,\theta) = (ze^{2\pi i\theta},\tau)$.
Consider the universal covering map $q:\RRR\to S^1$, $q(\theta)=e^{2\pi i \theta}$.
Since $\InitLift(s,t)$ preserves $\Yman=S^1\times 0$, and $\InitLift(s,0)=\id_{\Yman}$ for all $s\in \Sman$, the map
\[
 \DefInitLift:\Sman\times I \times \Yman\longrightarrow \Yman=S^1
\]
lifts to a function
\[
 \DefShFInitLift:\Sman\times I \times \Yman\longrightarrow\RRR,
\]
such that $\DefInitLift=q\circ\DefShFInitLift$ and $\DefShFInitLift(s,0,z)=0$ for $(s,z)\in \Sman\times \Yman$.
In other words,
\[
 \DefInitLift(s,t,z)=e^{2\pi i \DefShFInitLift(s,t,z)}.
\]
Since $q$ is a local diffeomorphism, and $\InitLift$ is continuous into $\Cinf$-topology of $\Ci{\Mman}{\Mman}$, it follows that the function 
\[
 \ShFInitLift:\Sman\times I \longrightarrow \Ci{\Yman}{\RRR},
\qquad 
\ShFInitLift(s,t)(z) = \DefShFInitLift(s,t,z)
\]
is continuous into $\Cinf$-topology of $\Ci{\Yman}{\RRR}$ as well.

Now take a $\Cinf$ function $\mu:J\to[0,1]$ satisfying (i) $\mu(0)=1$, (ii) $\mu=0$ outside $[-0.5,0.5] \cap J$, and define the map 
\[\hat\ShFInitLift:\Sman\times I \longrightarrow\Ci{S^1\times J}{\RRR}=\Ci{\Nman_{\Yman}}{\RRR}\]
by
\[
 \hat\ShFInitLift(s,t)(z,\tau)= \mu(\tau)\ShFInitLift(s,t)(z).
\]
Evidently, $\supp(\hat\ShFInitLift(s,t)) \subset \Int{\Nman_{\Uman}}$.
Therefore we can extend $\hat\ShFInitLift(s,t)$ by zero on all of $\Mman$, and so regard $\hat\ShFInitLift$ as a map $\hat\ShFInitLift:\Sman\times I \to\Ci{\Mman}{\RRR}$.
Similarly to the case (A) one can verify that $\hat\ShFInitLift(\Sman\times I)\subset\Gamma$ and $\hat\ShFInitLift$ has properties (a)-(c).
\end{proof}

\medskip

{\bf Case (B)}.
Suppose $\Yman$ is a non-orientable connected component of $\Xman$.
Denote 
\[
\tXman=\DoubleCover^{-1}(\Xman), 
\qquad 
\tNman_{\tYman}=\DoubleCover^{-1}(\Nman_{\Yman}),
\qquad 
\tYman=\DoubleCover^{-1}(\Yman).
\]
Let also $\tDifftMtX$ be the subgroup of $\tDifftM$ consisting of diffeomorphisms fixed on $\tXman$ (and commuting with $\Invol$).
Since $\InitHom(s_0,0)=\id_{\Mman}$, it follows that $\InitLift$ lifts to a map $\widetilde{\InitLift}: \Sman \times I \to \tDifftM$.

Notice that $\tYman$ is a connected orientable surface and $\tfunc$ satisfies axioms \AxBd-\AxFibr.
Therefore we can apply case (A) of Lemma~\ref{lm:exists_sh_func} and find a map 
\[
\hat\ShFInitLift:\Sman\times I \longrightarrow \Gamma=\{\afunc\in\Ci{\tMman}{\RRR}
\mid \FldA(\afunc)>-1\}
\]
satisfying the conditions (a)-(c): $\hat\ShFInitLift(s,0)=0$, $\supp\bigl(\hat\ShFInitLift(s,t)\bigr)\subset\Nman_{\hat\Yman}$, and $\hat\InitLift=\ShA\circ\hat\ShFInitLift(s,t)$ coincides with $\widetilde\InitLift(s,t)$ on $\hat\Yman$.

Notice that the restriction $\widetilde\InitLift(s,t)$ to $\tYman$ is a lifting of $\InitLift(s,t)|_{\Yman}$, and therefore it commutes with $\Invol$, that is
\[
 \widetilde\InitLift(s,t)\circ\xi(x) = \xi\circ\widetilde\InitLift(s,t)(x).
\]
Moreover, $\widetilde\InitLift(s,t)|_{\Yman} \in \widetilde{\Stab}_{\id}(\tfunc|_{\tYman})$.
Then it follows from the construction of Lemma~\ref{lm:exists_sh_func} that 
\[
\hat\ShFInitLift(s,t)|_{\tYman}\in 
\widetilde\Gamma_{\tYman} = \{\afunc\in\Ci{\tYman}{\RRR} \mid \FldA(\afunc)>-1, \ \afunc\circ\Invol=-\afunc \}. \]
In particular we have that
\begin{equation}\label{equ:hatShFInitLift_xi_skew_sym_on_tY}
 \hat\ShFInitLift(s,t)\circ\Invol(x)= -\hat\ShFInitLift(s,t)(x), \qquad \forall x\in\tYman.
\end{equation}

Now define two maps 
\begin{align*}
&\hat\ShFInitLift_1:\Sman\times I \to \Ci{\tMman}{\RRR}, &&
 \hat\ShFInitLift_1(s,t)= \frac{1}{2}\bigl(\hat\ShFInitLift(s,t)- \hat\ShFInitLift(s,t)\circ\Invol \bigr),
\\
 &\hat\InitLift_1=\ShA\circ\hat\ShFInitLift_1:\Sman\times I\to \Ci{\Mman}{\Mman},
&&
\hat\InitLift_1(s,t)(x)=\FlowA\bigl(x,\hat\ShFInitLift_1(s,t)(x)\bigr).                                                                       \end{align*}
\begin{lemma}\label{lm:properties_xi-maps}
$\hat\ShFInitLift_1$ and $\hat\InitLift_1$ have the following properties 
\begin{align}
\label{equ:hatShFInitLift_1=hatShFInitLift_on_tY}
 \hat\ShFInitLift_1(s,t)(x) &= \hat\ShFInitLift_1(s,t)(x), \qquad x\in\tYman \\
\label{equ:hatShFInitLift_1_xi_skewsym}
 \hat\ShFInitLift_1(s,t)\circ\Invol&=-\hat\ShFInitLift_1(s,t),\\
\label{equ:hatInitLift_1_xi_sym}
 \hat\InitLift_1(s,t)\circ\Invol &= \Invol\circ\hat\InitLift_1(s,t), \\
\label{equ:ShFInitLift_SI_wG}
\hat\ShFInitLift_1(\Sman\times I) &\subset \widetilde\Gamma,
\end{align}
and satisfy conditions {\rm(a)-(c)} of Lemma~{\rm\ref{lm:exists_sh_func}}.
Hence the map 
\[
\widetilde\ReqLift:\Sman\times I \longrightarrow \DiffMX,
\qquad
 \widetilde\ReqLift(s,t) = \hat\InitLift(s,t)_1^{-1} \circ \widetilde\InitLift(s,t)
\]
is a $\Invol$-equivariant lifting of $\InitLift(s,t)$ inducing a map $\ReqLift:\Sman\times I \to \DiffMX$ being a required lifting of $\InitHom$.
\end{lemma}
\begin{proof}
\eqref{equ:hatShFInitLift_1_xi_skewsym} is evident, \eqref{equ:hatShFInitLift_1=hatShFInitLift_on_tY} follows from \eqref{equ:hatShFInitLift_xi_skew_sym_on_tY}.
Moreover,
\begin{align*}
 \hat\InitLift_1(s,t)\circ\Invol(x)
 &=\FlowA\bigl(\Invol(x),\hat\ShFInitLift_1(s,t)\circ\Invol(x)\bigr) 
 =\FlowA\bigl(\Invol(x),-\hat\ShFInitLift_1(s,t)(x)\bigr)\\
  &=\Invol\circ\FlowA\bigl(x,\hat\ShFInitLift_1(s,t)(x)\bigr)
 =\Invol\circ\hat\InitLift_1(s,t)(x).
\end{align*}
which proves~\eqref{equ:hatInitLift_1_xi_sym}.

Finally, since $\hat\ShFInitLift(s,t)\in\Gamma$ we have that $\FldA\bigl(\hat\ShFInitLift(s,t)\bigr)>-1$ on all of $\Mman$.
Moreover, due to~\eqref{equ:hatShFInitLift_1_xi_skewsym} and the assumption that $\Invol^{*}\FldA=-\FldA$, it follows that 
$
 \FldA\bigl(\hat\ShFInitLift(s,t)\circ\Invol\bigr) = 
-\FldA\bigl(\hat\ShFInitLift(s,t)\bigr),
$
whence
\begin{align*}
 \FldA(\hat\ShFInitLift_1(s,t))&=
\frac{1}{2}\FldA\bigl(\hat\ShFInitLift(s,t) - \hat\ShFInitLift(s,t)\circ\Invol\bigr) =
\frac{1}{2}\Bigl[ \FldA\bigl(\hat\ShFInitLift(s,t)\bigr) + \FldA\bigl(\hat\ShFInitLift(s,t)\bigr)\Bigr] \\ &= \FldA\bigl(\hat\ShFInitLift(s,t)\bigr)>-1.
\end{align*}

Property (a) is evident, (b) follows from the relation $\Invol(\tNman_{\tYman})=\tNman_{\tYman}$, and (c) from~\eqref{equ:hatShFInitLift_1=hatShFInitLift_on_tY} and property (c) for $\hat\InitLift$.
Lemma~\ref{lm:properties_xi-maps} is finished.
\end{proof}

Thus for any connected component $\Yman$ we can change $\InitLift(s,t)$ on $\Nman_{\Yman}$ to make it a lifting of $\InitHom(s,t)$ fixed on $\Yman$.
Since neighbourhoods $\Nman_{\Yman}$ are disjoint, we can make these changes mutually on all of $\Nman$, and so assume that $\InitHom(s,t)$ is fixed on $\Xman$.
Then $\InitHom$ will be the desired lifting $\ReqLift$ of $\InitHom$.
Theorem~\ref{th:p_DiffMX_to_OrbfX_is_Serre} is completed.
\end{proof}

\def\cprime{$'$}
\providecommand{\bysame}{\leavevmode\hbox to3em{\hrulefill}\thinspace}
\providecommand{\MR}{\relax\ifhmode\unskip\space\fi MR }
\providecommand{\MRhref}[2]{%
  \href{http://www.ams.org/mathscinet-getitem?mr=#1}{#2}
}
\providecommand{\href}[2]{#2}

\end{document}